\documentclass[12pt,reqno,a4paper]{amsart}
\usepackage[utf8]{inputenc}
\usepackage{
amsmath,  amsfonts, amssymb,  amsthm,   amscd,
    gensymb,  graphicx, comment,  etoolbox, url,
    booktabs, stackrel, mathtools,enumitem, mathdots,  microtype, lmodern,    mathrsfs, graphicx, tikz,  longtable,tabularx, float, tikz, pst-node, tikz-cd, multirow, tabularx, amscd,  bm, array, makecell, diagbox, booktabs,ragged2e, caption, subcaption }
\usepackage[all]{xy}
\usepackage{makecell,slashbox}

\usepackage{xcolor}
\usepackage[utf8]{inputenc}
\usepackage{microtype, fullpage, wrapfig,textcomp,mathrsfs,csquotes,fbb}
\usepackage[colorlinks=true, linkcolor=blue, citecolor=blue, urlcolor=blue, breaklinks=true]{hyperref}
\usepackage[capitalise]{cleveref}
\usepackage{todonotes}
\usetikzlibrary{positioning}
\usetikzlibrary{shapes,arrows.meta,calc}

\usetikzlibrary{arrows}

\newtheorem{theorem}{Theorem}[section]

\newtheorem{corollary}[theorem] {Corollary}
\newtheorem{definition}[theorem]{Definition}
\newtheorem{example}[theorem]{Example}
\newtheorem{lemma}[theorem]{Lemma}

\newtheorem{proposition}[theorem]{Proposition}
\newtheorem{remark}[theorem]{Remark}

\newcommand\R{\mathbb{R}}
\newcommand\Z{\mathbb{Z}}

\newcommand{\TC}{\mathrm{TC}}

\newcommand{\cat}{\mathrm{cat}}
\newcommand{\wgt}{\mathrm{wgt}}

\newcommand{\gen}{\mathfrak{genus}}

\newcolumntype{x}[1]{>{\centering\arraybackslash}p{#1}}

\title{Higher topological complexity of Seifert fibered manifolds}

\author{Navnath Daundkar}
\address{Department of Mathematics, Indian Institute of Technology Madras, Chennai, India.}
\email{navnath@iitm.ac.in}
\author{Rekha Santhanam}
\address{Department of Mathematics, Indian Institute of Technology Bombay, India}
\email{reksan@iitb.ac.in}
\author{Soumyadip Thandar}
\address{Statistics and Mathematics unit, Indian Statistical Institute, Kolkata.}
\email{stsoumyadip@gmail.com}

\begin{document}
\begin{abstract}
In this article, we investigate the higher topological complexity of oriented Seifert fibered manifolds that are Eilenberg--MacLane spaces $K(G,1)$ with infinite fundamental group $G$.  We first refine the cohomological lower bounds for higher topological complexity by introducing the notion of higher topological complexity weights. As an application, we show that the $r^{\text{th}}$ topological complexity of these manifolds lies in $\{3r-1, 3r, 3r+1\}$, and characterize large families where the value is $3r$ or $3r+1$. Additionally, we establish a sufficient condition for higher topological complexity to be exactly $3r$ when the base surface is orientable and aspherical. Finally, we show that the higher topological complexity of the wedge of finitely many closed, orientable, aspherical $3$-manifolds is exactly $3r+1$.
\end{abstract}

\maketitle

\section{Introduction}
Schwarz \cite{Svarc61} introduced the notion of \emph{genus} of a fibration. Given a fibration $p:E\to B $, where $E$ and $B$ are path-connected spaces, the \emph{genus} of  $p$ is defined as the minimum number of open sets $\{U_1, \ldots, U_k\} $ covering $B$ such that $p$ has local continuous sections on each open set $U_i$. This number is denoted as $\gen(p)$.
For a path connected space $X$, consider the fibration
\begin{equation}\label{eq: tcn fibration}
e_{r}: X^I\to X^r  \text{ defined by }  e_{r}(\gamma)=\bigg(\gamma(0), \gamma(\frac{1}{r-1}),\dots,\gamma(\frac{r-2}{r-1}),\gamma(1)\bigg).     
\end{equation}

The \emph{higher topological complexity} \cite[Definition 3.1]{RUD2010}, denoted as $\TC_r(X)$, is defined to be the genus of $e_{r}$.
For $r=2$, this notion was introduced by Farber in \cite{FarberTC} and is known as the topological complexity of a space. The topological complexity of $X$ is denoted by $\TC(X)$. Farber proved that the topological complexity of a space is a homotopy invariant. 

The higher topological complexity of a space $X$ is closely related to its \emph{Lusternik-Schnirelmann category} (\emph{LS-category}), denoted by $\mathrm{cat}(X)$, which is the smallest integer $r$ such that $X$ can be covered by $r$ open subsets $V_1, \dots, V_r$, where the inclusion $V_i\xhookrightarrow{} X$ is null-homotopic for each $i$.
In particular, it was proved in  \cite{RUDIY2014} that,
$\mathrm{cat}(X^{r-1})\leq \TC_r(X)\leq \mathrm{cat}(X^r).$
The higher topological complexity of closed manifolds of dimension $1$ and $2$ is completely known (for reference see \cite[Proposition 5.1]{htcsurfaces}. 
On the other hand, it was shown in 
\cite{cat3mfds} that, the LS-category of closed $3$-manifolds depends only on the fundamental group.
In this article, we explore the question of computing higher topological complexity of orientable $3$-manifolds. 
In particular, we consider orientable Seifert manifolds which are aspherical. 

Seifert introduced the notation for Seifert fibered manifolds in \cite{Seifert1933}. The notation  $M_O:=(O,o,g \mid e: (a_1,b_1),\dots, (a_m,b_m))$ describes an orientable Seifert manifold with an orientable orbit surface $\Sigma_g$
of genus $g$. The notation $M_N:=(O,n,g \mid e: (a_1,b_1),\dots, (a_m,b_m))$ denotes an orientable Seifert manifold with a non-orientable orbit surface $N_g$. Here  $e$ is the Euler number, $m$ is the number of singular fibers and, for each $i$, $(a_i,b_i)$ is a pair of relatively prime integers that characterize the twisting of the $i^{th}$ singular fiber.

Scott in \cite[Lemma 3.1]{Scott83} shows that if a compact orientable $3$-manifold $M$ has infinite fundamental group with universal cover $S^2\times \R$, then $M$ is either $S^2\times S^1$ or $\R P^3\#\R P^3$. Therefore, it follows from  \cite[Proposition 1]{Surveysf} that, any Seifert fibered manifold with infinite fundamental group is a $K(\pi,1)$ space with the exception of  $S^2\times S^1$ and $\R P^3\#\R P^3$. It is easy to see that $\TC(S^2\times S^1)=4$ and also, from \cite{COEHN2019} it follows that $\TC(\R P^3\# \R P^3)=7$. The $\rm{mod}~p$ cohomology ring for oriented aspherical Seifert manifolds with infinite fundamental group was computed by Bryden and Zvengrowski \cite{CoringSiefert}.  A straight forward computation gives that the LS-category of these manifolds is $4$.

In \cite[Theorem 6.6]{MescherSC}, Mescher has shown that if a closed oriented $3$-manifold $M$ that
is not a
rational homology $3$–sphere, and it is not dominated by any closed oriented product
manifold, then $\TC(M)\in \{5,6,7\}$.  Mescher also observed that the total space of a circle bundle over a closed oriented surface of
positive genus with non-zero Euler number satisfies the conditions of this result. 
For this class of manifolds, Neofytidis (\cite{neofytidis2022}) has established a stronger result. In \cite[Page 16]{neofytidis2022}, Neofytidis has shown that the topological complexity of a total space of circle-bundle over $\Sigma_g$ with $g>1$ (these are hyperbolic surfaces) is $6$. 

Oriented Seifert fibered manifolds are natural generalizations of circle bundles over surfaces (both can be orientable and non-orientable). While every circle bundle over a surface is a Seifert fibered manifold, the class of Seifert fibered manifolds is significantly broader (see Example \ref{example}). These manifolds allow for more intricate fiber structures, including singular fibers with non-trivial local models, making the Seifert fibered category a rich and expansive extension beyond the realm of regular circle bundles.

In this article,  we give sharp bounds on the higher topological complexity of aspherical Seifert manifolds satisfying some minor conditions. 
Before the computations, we introduce the notion of higher topological complexity weights to improve the Rudyak's cohomological lower bound \cite[Proposition 3.4]{RUD2010} on the higher topological complexity by proving a general version of \cite[Theorem 6]{GrantCweights} (see also \Cref{thm: c wt2}). 
Using this for an orientable Seifert-fibered manifold $M$, we prove that $\TC_r(M)\in \{3r-1, 3r, 3r+1\}$, and explicitly determine the subclass of Seifert-fibered manifolds for which $\TC_r(M)\in \{3r,3r+1\}$ (see \Cref{thm: TCofSeifert} and \Cref{prop:n2equals1higherTC}). In Proposition \ref{lem:tck-MO-ub}, we show that under certain condition the topological complexity of these manifolds is bounded above by $3r$, and, we also observe that the topological complexity of a aspherical oriented Seifert $3$-manifold $M$, with $\pi_1(M)$ being Heisenberg group, is bounded by $3r$ (see Remark \ref{remk:heisen}).  Additionally, in Corollary \ref{cor: equality}, we show that $\TC_r(M_O)=3r$, when the base surface is orientable and aspherical. Finally, in Proposition \ref{prop: higher tc of iterated wedge} we show that the higher topological complexity of the wedge of finitely many closed, orientable, aspherical $3$-manifolds is exactly $3r+1$.

It is not straightforward to verify whether Seifert fibered manifolds are not dominated by a closed oriented product manifold, which is one of the important conditions of Mescher's result \cite[Theorem 6.6]{MescherSC}. A known nontrivial result in this direction was given by Kotschick and Neofytidis in \cite[Lemma 1]{Kotschick-Neofytidis}. 
They showed that, if $N\to F$ is any  oriented circle bundle with non-zero Euler number over a closed aspherical surface, then every continuous map $\Sigma_g\times S^1\to N$ has degree zero, where $g\geq 1$. In other words, the total spaces of a circle bundle over closed oriented surfaces of positive genus with nonzero Euler number are not dominated by any closed oriented product manifold. Consequently, this gives a class of examples that satisfy Mescher's conditions and have topological complexity $\in\{5, 6 , 7\}$. We not only establish that aspherical Sefiert manifolds have topological complexity $\in \{5,6,7\}$ but  give sharper bounds for a large class of aspherical Seifert manifolds in \Cref{thm: TCofSeifert}, \Cref{prop:n2equals1higherTC} and Corollary \ref{cor: equality}.

\section{Background}

\subsection{Weights of cohomology classes}
The notion of category weight of a cohomology class was introduced by Fadell and Husseini in \cite{cwtFadellHusseini}  to give a lower bound on the LS-category of a space. This was generalized by Farber and Grant  to the weight of a cohomology class with respect to any fibration  (see \cite{GrantCweights,GrantSymm}) to get a lower bound on the topological complexity of a space.

\begin{definition}\label{def: weight}
Let $u\in H^{\ast}(B;G)$ be a cohomology class, where $G$ is an abelian
group. The weight of $u$ with respect to $p:E\to B$, is defined to be the largest integer $\wgt_p(u)$ as
\[\wgt_p(u)=\mathrm{max}\{k : f^{\ast}(u)=0\in H^{\ast}(Y;G) \text{ for all } f:Y\to B
\text { with } \gen(f^{\ast}p)\leq k\},\] where $f^{\ast}p$ is the pullback fibration corresponding to $f$.
\end{definition}

In \cite{GrantSymm}, Farber and Grant improved Schwarz's original cohomological lower bound (see \cite{Svarc61}) on the genus of a fibration. In particular, they proved the following result.
\begin{proposition}[{\cite[Proposition 32, Proposition 33]{GrantSymm}}]\label{prop:genuslowerbd}
Let $p:E\to B$ be a fibration and $u_i \in H^{d_i}(B; G_i)$ be cohomology classes, $i = 1,\dots,l$, such that their cup-product $\prod_{i=1}^l u_i\in   H^d(B; \otimes_{i=1}^l G_i)$ is non-zero, where $d = \sum _{i=1}^ld_i$. Then
\[\gen(p)>\wgt_p(\prod_{i=1}^lu_i)\geq \sum_{i=1}^l\wgt_p(u_i).\]
\end{proposition}

Let $X^I$ be a free path space with compact open topology. 
Consider the free path space fibration  $e_2:X^I\to X^2$ defined by \[e_2(\gamma):=(\gamma(0),\gamma(1)).\] Then $\gen(e_2)=\TC(X)$ (for reference see \cite{FarberTC}). 
Let $u\in H^{\ast}(X)$ be a non-zero cohomology class and $\bar{u}=u\otimes 1-1\otimes u$ be the corresponding zero divisor, then it is known from \cite[Proposition 30]{GrantSymm} that $\wgt_p(\bar{u})\geq 1$.
Therefore, if we have cohomology classes $u_i\in H^{\ast}(X)$ such that $\prod_{i=1}^{l}\bar{u_i}\neq 0$ with $\wgt_{e_2}(\bar{u_i})\geq 2$ for some $i$'s, then \Cref{prop:genuslowerbd} improves the well-known lower bound on $\TC(X)$ given by zero-divisor cup-length.

The following theorem is crucial in finding cohomology classes whose TC-weight is at least two.
\begin{lemma}[{\cite[Lemma 4]{GrantCweights}}]
Let $e_2:X^I\to X^2$  be a free path space fibration and $f = (\phi,\psi): Y \to X\times X$ be a map where $\phi,\psi$ denote the projections of $f$ onto the first and second factors of $X\times X$, respectively. Then  $\gen(f^{\ast}e_2) \leq k$ if and only if $Y=\cup_{i=1}^kA_i$, where $A_i$ for $i=1,\cdots k$ are open in $Y$ and $\phi|_{A_i}\simeq \psi|_{A_i}:A_i\to X$ for any $i=1,\dots ,k$.
\end{lemma}

Next, we recall the results of Farber and Grant describing zero divisors with TC-weights at least two. We begin with the definition of stable cohomology operation. Let $R$ and $S$ be abelian groups.
\begin{definition}
A degree $d$ stable cohomology operation 
$\mu:H^{\ast}(-;R)\to H^{\ast+d}(-;S)$
is a family of natural transformations $\mu:H^n(-;R)\to H^{n+d}(-;S)$,  for $n\in \mathbb{Z}$ which commutes with the suspension isomorphisms. 
\end{definition}
It follows that the stable cohomology operation $\mu$ commutes with all connecting homomorphisms in Mayer-Vietoris sequences and each
homomorphism $\mu$ is a group homomorphism.
The excess of a stable cohomology operation $\mu$ is the positive integer denoted as $e(\mu)$ and defined as $e(\mu):=1+\max\{k \mid \mu(u)=0 \text{ for all } u\in H^{k}(X;R)\}.$

Let $p_i:X\times X\to X$ be the projection onto the $i^{\text{th}}$ factor of $X\times X$ for $i=1,2$. Note that $\bar{u}=p_2^{\ast}(u)-p_1^{\ast}(u)$.
Let $\mu$ be any stable cohomology operation. Then, using naturality and additivity of $\mu$, 
we get
$$\mu(\bar{u})=\mu(1\otimes u-u\otimes 1)=\mu(p_2^{\ast}(u)-p_1^{\ast}(u))=p_2^{\ast}(\mu(u))-p_1^{\ast}(\mu(u))=\overline{\mu(u)}.$$

The following result shows how to construct zero-divisors with TC-weights at least $2$.
\begin{theorem}[\label{thm:wgttc>2}{\cite[Theorem 6]{GrantCweights}}]
 Let $\mu: H^{\ast}(-;R)\to H^{{\ast}+i}(-;S)$ be the stable cohomology operation of degree $i$ and $e(\mu)\geq n$. Then for any $u\in H^n(X;R)$, the TC-weight $\wgt_{e_2}(\overline{\mu(u)})\geq 2$. 
\end{theorem}

We use  \Cref{thm:wgttc>2} to compute the lower bound on the topological complexity of oriented Seifert fibered manifolds.
\subsection{Cohomology of Seifert fibered manifolds}

 We recall the results related to the $\rm{mod}~ p$ cohomology ring of Seifert fibered manifolds from \cite{CoringSiefert}. We first set up the notation required for the same.

\begin{enumerate}
\item For any prime $p$, assume without loss of generality that $a_1,\dots, a_{n_p} \equiv 0 (\rm {mod } ~p)$ and $a_{n_{p+1}},\dots,a_m \not\equiv 0 (\rm {mod } ~p)$. 
In this case there exist integers $a_1', \dots, a_{n_p}'$ , such that $a_1 = pa_1' ,\dots,a_{n_p} = pa_{n_p}'$. 
\item We have $(a_i,b_i)=1$ and  $c_i, d_i\in \mathbb{Z}$ such that $a_id_i-b_ic_i=1$ for $1\leq i\leq m$.  
Then for $1\leq i\leq n_p$, $b_i, c_i\not\equiv 0 (\rm {mod}~ p)$.

\item When $n_p=0$, that is $a_i\not\equiv 0 (\rm {mod}~ p)$ for all $1\leq i\leq m$, let $r$ be a positive integer such that $b_i\equiv 0 (\rm{mod}~ p)$ and $b_{r+j}\not\equiv 0 (\rm{mod}~ p)$ for $1\leq j\leq m-r$. For $1\leq i\leq r$, there exist $b_i'$ such that $b_i=pb_i'$. 
Let $A=\prod_{i=1}^{m}a_i$, $A_i=A/a_i\in \mathbb{Z}$ and $C=\sum_{i=1}^{m}b_iA_i$. Note that $A\not\equiv 0(\rm{mod} ~p)$ when $n_p=0$. 
\end{enumerate}

We note that without loss of generality in \cite{CoringSiefert}, the notation $n$ is used for $n_p$. We will use $n_p$ to clarify which $p$ we are discussing in our notation.   

 \begin{proposition}[{\cite[Theorem 1.1]{CoringSiefert}}\label{thm: cring1}]
Let $\delta_{jk}$ denote the Kronecker delta. Suppose $M_O= (O,o;g ~|~ e: (a_1,b_1),\dots,(a_m,b_m))$.
If $n_p > 0$, then as a graded vector space,
$$H^{\ast}(M_O;\mathbb{Z}_2)=\mathbb{Z}_2\{1,\alpha_i,\theta_l,\theta_l',\beta_i,\phi_l,\phi'_l,\gamma~|~2\leq i\leq n_p, ~~1\leq l\leq g\}$$
where $|\alpha_i|=|\theta_l|=|\theta'_l|=1$, $|\beta_i|=|\phi_l|=|\phi_l'|=2$, and $|\gamma|=3$. 
Let $\beta_1 = - \sum_{i=2} ^{n_p} \beta_i$. The non-trivial cup products in $H^{\ast}(M_O;\Z_p)$ are given by:
\begin{enumerate}
    \item For $p=2$ and $2\leq i$, $j\leq n_2$,
$\alpha_i\cdot\alpha_j=\binom{a_1}{2}\beta_1+\delta_{ij}\binom{a_i}{2}\beta_i.$
Moreover, if $2\leq k\leq n_2$ as well, then
\[\alpha_i\cdot\alpha_j\cdot\alpha_k=\binom{a_1}{2}\gamma \hspace{0.5cm} ~~\text{ if } i\neq j \text{ or } j \neq k,~~ \text{and}
~~\alpha_i^3=\bigg[\binom{a_1}{2}+\binom{a_i}{2}\bigg]\gamma.\]
\item For any prime $p$, ~$2\leq j\leq n_p$ and $1\leq l\leq g$ we have,
$\alpha_j\cdot\beta_j=-\gamma \hspace{0.5cm}\text{and} \hspace{0.5cm}\theta_l\cdot\phi_l=-\gamma.$
In addition, the mod-$p$ Bockstein on $H^1(M_O, \Z_p)$ is given by: 
$$B_p(\alpha_j)=-a'_jc_j\beta_j+a'_1c_1\beta_1,~~ B_p(\theta_l)=0.$$
\end{enumerate}
\end{proposition}

\begin{remark}\label{remk:main123}
For both Seifert-fibered manifolds of type $M_O$ and $M_N$, from \cite[Remark 1.2 and Remark 1.5]{CoringSiefert} for $p = 2$ we get, $B_2(\alpha_i)=\alpha^2_i$.
   
\end{remark}

\begin{proposition}[{\cite[Theorem 1.4]{CoringSiefert}}\label{thm: cring 3}]
Suppose $M_N=(O,n;g ~|~ e: (a_1,b_1),\dots,(a_m,b_m))$. If $n_p > 0$, then as a graded vector space,
$H^{\ast}(M_N;\Z_p)=\Z_p\{1,\alpha_i,\theta_l,\beta_i,\phi_l,\gamma ~~: ~~2\leq i\leq n_p,~1\leq l\leq g\}.$
Let $\beta_1 = - \sum_{i=2} ^{n_p} \beta_i -2\sum_{j=1}^g\phi_j$. The non-trivial cup products in $H^{\ast}(M_N;\Z_p)$ are given by:
\begin{enumerate}
    \item For $p=2$ and $2\leq i$, $j\leq n_2$,
$\alpha_i\cdot\alpha_j=\binom{a_1}{2}\beta_1+\delta_{ij}\binom{a_i}{2}\beta_i$.
Moreover, if $2\leq k\leq n_2$ as well, then
\[\alpha_i\cdot\alpha_j\cdot\alpha_k=\binom{a_1}{2}\gamma \hspace{0.5cm}~~ \text{ if } i\neq j \text{ or } j \neq k,~~\text{and} 
~~\alpha_i^3=\bigg[\binom{a_1}{2}+\binom{a_i}{2}\bigg]\gamma.\]
\item For any prime $p$, ~$2\leq j\leq n_p$ and $1\leq l\leq g$ we have,
$$\alpha_j\cdot\beta_j=-\gamma \hspace{0.5cm}\text{and} \hspace{0.5cm}\theta_l\cdot\phi_l=-\gamma.$$ 
In addition, the mod-$p$ Bockstein on $H^1(M_N, \Z_p)$ is given by: 
$B_p(\alpha_j)=-a'_jc_j\beta_j+a'_1c_1\beta_1,~~ B_p(\theta_l)=0.$
\end{enumerate}
\end{proposition}

When $n_p=0$, the cohomology ring for Seifert manifolds of type $M_O$ is computed in \cite{CoringSiefert}, which we describe now.
\begin{proposition}[{\cite[Theorem 1.3]{CoringSiefert}}\label{thm: cring2}]
Suppose $M_O= (O,o;g~ |~ e: (a_1,b_1),\dots,(a_m,b_m))$ and $n_p = 0$. Define $r$ so that $b_i \equiv 0 \text{(mod p)}$ for $1\leq i\leq r$, $b_i\not\equiv 0 \text{(mod p)}$ for $r+1\leq i\leq m$. When $Ae+C \not\equiv 0 \pmod{p}$, $H^{\ast}(M_O;\mathbb{Z}_p)$ has generators $\theta_l,\theta'_l,\phi_l,\phi'_l$, $1\leq l \leq g$ (as described in \Cref{thm: cring1}). If $Ae+C \equiv 0 \text{(mod p)}$, then 
$$H^{\ast}(M; \mathbb{Z}_p)=\mathbb{Z}_p\{1,\alpha,\theta_l,\theta_l',\beta,\phi_l,\phi'_l,\gamma~|~1\leq l\leq g\},$$
where $|\alpha|=1$ and $|\beta|=2$. 
The non-trivial cup products are given by: 
\begin{enumerate}
\item For $p=2$, if $Ae+C \equiv 0 \text{(mod 2)}$,
$\alpha^2=\bigg[q+\frac{1}{2}(Ae+C)\bigg]\beta$,
where $q$ is defined to be number of $b_i$, $1\leq i\leq r$, which are congruent to $2$ (mod $4$). 
\item If $Ae+C\equiv 0 \text{(mod p)}$ then for $1\leq l\leq g$, 
\[\alpha\cdot\theta_l=\phi_l, 
~~\alpha\cdot\theta'_l=\phi_l',
~~\theta_l\cdot\theta_l'=\beta,
~~\alpha\cdot\beta=-\gamma, ~~\theta_l\cdot \phi_l'=\theta_l'\cdot \phi_l=\gamma.\]
Moreover, the mod-$p$ Bockstein on $H^1(M_O;\Z_p)$ is given by:
$$B_p(\alpha)=-A^{-1}\bigg[\sum_{i=1}^rb'_iA_i+\frac{Ae+C}{p}\bigg]\beta\in H^2(M_O;\mathbb{Z}_p), ~~~ B_p(\theta_l)=B_p(\theta_l')=0.$$
\end{enumerate}
\end{proposition}

\section{Higher cohomology weights}
We begin by defining the a higher analogue of the $\TC$-weight. Recall the Definition \ref{def: weight} of the weight of a cohomology class with respect to the fibration. 
\begin{definition}
The $\TC_r$-weight of a cohomology class $u\in H^{\ast}(X^r)$ with respect to the fibration $e_r$ is defined as the $\wgt_{e_r}(u)$.
\end{definition}

In this section, we show how to obtain cohomology classes in the kernel of $d_r^{\ast}: H^{\ast}(X^r)\to H^{\ast}(X)$ whose $\TC_r$-weights are at least $2$. We then use this to compute the higher topological complexity of Seifert fibered manifolds. 
\begin{lemma}\label{lem: htc gen}
Let $f: Y\to X^r$ be a continuous map with $f=(\phi_1,\dots,\phi_r)$, where $\phi_i$'s are the projection of $f$ onto the $i$th factor of $X^r$. Then $\gen(f^{\ast}e_r)\leq k$ if and only if there exist an open cover $\{U_1,\dots,U_k\}$ of $Y$ such that $\phi_i|_{U_j}\simeq \phi_l|_{U_j}$ for all $1\leq i,l\leq r$ and $1\leq j\leq k$.
\end{lemma}

\begin{proof}
Suppose that  $\gen(f^{\ast}e_r)\leq k$. Then there exists an open cover  $\{U_1,\dots,U_k\}$ of $Y$ and sections $s_j: U_j\to f^{\ast}X^{I}$ of $f^{\ast}e_r$ for $1\leq j\leq k$. 
Let $\Tilde{f}:f^{\ast}X^I=Y\times_f X^I\to X^I$ be the projection onto the second factor. For $1\leq j\leq k$, the commutativity of a diagram in \eqref{eq: comm diagram1} shows 
\begin{equation}\label{eq: comm diagram1}
 \xymatrix{
f^{\ast}X^I\ar[d]_{f^{\ast}e_r} \ar[r]^{\Tilde{f}}& X^I\ar[d]^{e_r}\\
Y \ar[r]^{f} & X^r \\ 
}   
\end{equation}
$e_r\circ \Tilde{f}\circ s_j=f|_{U_j}$. Therefore, for any $a\in U_j$, we have, 
\[\bigg(\Tilde{f}\circ s_j(a)(0),\Tilde{f}\circ s_j(a)(\frac{1}{r-1}),\dots, \Tilde{f}\circ s_j(a)(\frac{r-2}{r-1}),\Tilde{f}\circ s_j(a)(1)\bigg)=(\phi_1(a),\dots, \phi_{r}(a)).\]
We now define a homotopy $F^i_j:U_j\times I\to X$ between $\phi_{i}|_{U_j}$ and $\phi_{i+1}|_{U_j}$ by \[F^i_j(a,t)=\Tilde{f}\circ s_j(a)(\frac{t+i-1}{r-1}),\] for $1\leq i\leq r-1$. Note that, $F^i_j(a,0)=\phi_{i}|_{U_j}$ and $F^i_j(a,1)=\phi_{i+1}|_{U_j}$ for $1\leq j\leq k$. Since being homotopic is a transitive property, it follows that $\phi_i|_{U_j}\simeq \phi_l|_{U_j}$ for any $1\leq i, l\leq r$.

Conversely, let $f=(\phi_1,\dots,\phi_r):Y\to X^r$ be such that there is an open cover $\{U_1,\dots,U_k\}$ of $Y$ such that $\phi_i|_{U_j}\simeq \phi_l|_{U_j}$ for all $1\leq i,l\leq r$ and $1\leq j\leq k$. Let $G_i:U_j\times I\to X$ be a homotopy between $\phi_i|_{U_j}$ and $\phi_{i+1}|_{U_j}$. We define sections on $U_j$ for, $1\leq j\leq k$, via $G_i$'s. 
\vspace{-3mm}
\begin{equation}\label{eq: diag2}
 \xymatrix{
U_j\times_{f} X^I \ar[r]^{\Tilde{f}}& X^I\\
U_j\ar[ru]_{g_j}\ar[u]^{s_j} &  \\ 
}   
\end{equation}
It follows from the diagram in
\eqref{eq: diag2} that any section $s_j:U_j\to f^{\ast}X^I=Y\times _f X^I$ has a form $s_j(a)=(a,g_j(a))$, where $g_j:U_j\to X^I$ for $1\leq j\leq k$. Therefore, to define sections $s_j$, it is enough to define maps $g_j$'s.
We partition the interval $[0,1]$ into smaller intervals $ [\frac{n-1}{r-1},\frac{n}{r-1}]$ for $1\leq n\leq r$ and define $g_j$ as follows:
\begin{equation}\label{eq: gj} g_j(a)(t)=G_{n}(a, (r-1)t-n+1), ~~ t\in \bigg[\frac{n-1}{r-1},\frac{n}{r-1}\bigg].
\end{equation}

One can observe that $e_r\circ g_j=f|_{U_j}$. Therefore, $s_j:U_j\to f^{\ast}X^I$ is a well-defined section of $f^{\ast}e_r$ for $1\leq j\leq k$. Thus, $\gen(f^{\ast}e_r)\leq k$.
\end{proof}

We now explain how to find cohomology classes in $\ker( d_r^{\ast})$ with $\TC_r$-weights at least $2$.
Let $v\in H^m(X;R)$ be a non-zero class and $p_i$ is the projection of $X^r$ onto the $i^{th}$ factor. Then we  observe that $\bar{v}_{ij}=p_{j}^{\ast}(v)-p_i^{\ast}(v)\in \ker(d_r^{\ast})$. If $\mu$ is a stable cohomology operation, then \[\mu(\bar{v}_{ij})=\mu(p_{j}^{\ast}(v)-p_i^{\ast}(v))=p_{j}^{\ast}(\mu(v))-p_i^{\ast}(\mu(v))=\overline{\mu(v)_{ij}}.\]
\begin{theorem}\label{thm: c wt2}
 Let $\mu: H^\ast(-;R)\to H^{\ast+d}(-;S)$ be a stable cohomology operation of degree $d$  and $e(\mu)$ be of its excess such that $e(\mu)\geq m$. Then for any $v\in H^m(X;R)$, \[\wgt_{e_r}(\mu(\bar{v}_{ij}))\geq 2\] for $1\leq i,j\leq n$. 
\end{theorem}
\begin{proof}
Let $f=(\phi_1,\dots,\phi_r):Y\to X^r$ be continuous function with $\gen(f^\ast e_r)\leq 2$. Then using \Cref{lem: htc gen}, we have $Y=U\cup V$ such that $\phi_i|_{U}\simeq \phi_{j}|_{U}$ and $\phi_i|_{V}\simeq \phi_{j}|_{V}$ for all $1\leq i, j\leq r$. 
We need to show that for any $v\in H^{m}(X;R)$, we have $f^\ast(\overline{\mu(v)_{ij}})=0$ for $1\leq i,j\leq r$. 
Now consider the Mayer-Vietoris sequence for $Y=U\cup V$
\[\cdots\to H^{m-1}(U\cap V;R)\stackrel{\delta}{\to}H^m(Y;R)\to
H^m(U;R)\oplus H^m(V;R)\to\cdots\]

Observe that $f^{\ast}(\bar{v}_{ij})=\phi_{j}^{\ast}(v)-\phi_{i}^{\ast}(v)$. 
Since $\phi_j|_{U}\simeq \phi_{i}|_{U}$ and $\phi_j|_{V}\simeq \phi_{i}|_{V}$ for all $1\leq i,j\leq r$, there exists $u\in H^{m-1}(U\cap V;R)$ such that $\delta(u)=f^{\ast}(\bar{v}_{ij})$. 
Therefore, \[f^{\ast}(\overline{\mu(v)_{ij}}) = f^{\ast}(\mu(\bar{v}_{ij}))=\mu(f^{\ast}(\bar{v}_{ij}))=\mu(\delta(u))=\delta(\mu(u))=0,\] since the boundary homomorphism $\delta$ commutes with $\mu$ and $e(\mu)\geq m$. This implies $\wgt_{e_r}(\mu(\bar{v}_{ij}))\geq 2$ for $1\leq i,j\leq r$.  
\end{proof}

\begin{remark}
The proofs of \Cref{lem: htc gen} and \Cref{thm: c wt2} are based on the similar techniques used in \cite[Lemma 4]{GrantCweights} and \cite[Theorem 6]{GrantCweights}. 
\end{remark}

\section{Higher topological complexity of Seifert fibered manifolds}
In this section, we show that in most cases the higher topological complexity $\TC_r$ of aspherical Seifert fibered manifolds is $3r$ or $3r+1$. The bounds are obtained by using cohomological lower bound, and we get an exact value of higher topological complexity for these manifolds which satisfy the assumptions of Proposition \ref{lem:tck-MO-ub}.
In \Cref{prop:n2equals0higherTC}, for an another class of Seifert manifolds, their topological complexity lies  between $3r-1$ and $3r+1$.  Here, we need to use the cohomology classes  with higher weights to compute the lower bound.

The non-higher analogue of the following lemma and its proof are essentially contained in \cite[Corollary 3.8]{grantfibsymm}. We adopt those arguments in the higher setting for the sake of completeness.
Throughout this section, we denote $G$ as a finitely generated group, $Z=Z(G)$ as its centre, and $d_r:G\to G^r$, for $r\geq 2$ as the diagonal map.

\begin{proposition}\label{lemma: nilpotent torsion free}
Let $G$ be a torsion free nilpotent group. Then the cohomological dimension of $\frac{G^r}{d_r(Z)}$ is finite.  
\end{proposition}
\begin{proof}
Since $G$ is a finitely generated torsion-free nilpotent group, there exists a central series
\[
G = G_0 \geq G_1 \geq \cdots \geq G_{n-1} = Z(G) \geq G_n = \{1\},
\]
where each quotient $G_i/G_{i+1}$ is free abelian, the sum of whose ranks equals the rank of $G$. Hence $\operatorname{cd}(G) = \operatorname{rk}(G)$, and therefore such groups always have finite cohomological dimension.

Suppose $H=\frac{G^r}{d_r(Z)}$.
To show that it has finite cohomological dimension, it suffices to show that it is torsion-free and nilpotent. As the class of torsion-free nilpotent groups $\mathcal{N}$ is closed under finite direct products and subgroups, both $d_r(Z)$ and $G^r$ lie in $\mathcal{N}$. Since the quotient of a nilpotent group by a nilpotent subgroup is again nilpotent, it follows that $H$ is nilpotent. Thus, it remains to show that $H$ is torsion-free.

Let $[(g_1,\dots,g_r)] \in H$ be of finite order. Then each $[g_i] \in G/Z$ has finite order for $1 \le i \le r$. Therefore, it is enough to prove that $G/Z$ is torsion-free. Fixing $1 \le i \le r$, denote $g_i$ simply by $g$.

Suppose, for contradiction, that $G/Z$ is not torsion-free. Then there exists $0 \neq gZ \in G/Z$ such that $g^n \in Z$ for some $n > 0$. Thus, for all $h \in G$ we have $g^nh = hg^n$. This implies $(hgh^{-1})^n = g^n$. By \cite[page 30, 2.1.2]{Lennox}, torsion-free nilpotent groups have the \emph{unique root property}, i.e., if $x^n = y^n$, then $x = y$. Hence $hgh^{-1} = g$ for all $h \in G$, which implies $g \in Z$, contradicting the assumption that $gZ \neq Z$. Therefore, $G/Z$ is torsion-free.

Hence $H$ is torsion-free and nilpotent, and the result follows.
\end{proof}

In what follows we obtain an upper bound on the higher topological complexity of aspherical Seifert fibered manifolds of type $M_O$ under certain conditions.
We begin by proving a higher analogue of \cite[Proposition 3.7]{grantfibsymm}, which will help us to obtain sharp upper bound on the higher topological complexity of aspherical manifolds.

\begin{proposition}\label{prop:tcnub}
Let $G$ be a torsion-free discrete group. Then
    $\TC_r(K(G,1))\leq \cat(G^r/d_r(Z))$.
\end{proposition}
\begin{proof}
The proof is a straightforward generalization of the proof of  \cite[Proposition 3.7]{grantfibsymm} for the higher topological complexity using  \cite[Lemma 2.4]{daundkar2023group} and \cite[Theorem3.1]{daundkar2023group}.
\end{proof}

\begin{proposition}\label{lem:tck-MO-ub}
Suppose  $G=\pi_1(M_O)$ with $\mathrm{cd}(G^r/d_r(Z))$ is finite. Then  $$\TC_r(M_O)\leq 3r.$$    
\end{proposition} 

\begin{proof}
It follows from \cite[Page 5]{Surveysf} that $Z$ contains the infinite cyclic group, that is, the set of integers $\Z$ up to the isomorphism.    Then using \Cref{prop:tcnub} we get the inequality 
$$\TC_r(M_O)\leq \cat\left(\frac{G^r}{d_r(Z)}\right).$$ 
Suppose $H=\frac{G^r}{d_r(Z)}$ and  $1$ denotes the identity of $G$. Then, we have a short exact sequence:
\[\begin{tikzcd}
   1 \arrow{r}& Z\arrow{r}{d_r|_{Z}} & G^r\arrow{r}{q}& H \arrow{r} & 1.
\end{tikzcd}\]
Since $\mathrm{cd}(H)$ is finite, from \cite[Theorem 5.5 (i)]{Bieri} we obtain the equality $\mathrm{cd}(H)= \mathrm{cd}(G^r)-\mathrm{cd}(Z)$.  
Moreover,  $K(G^r,1)\simeq M_O^r$ as $M_O$ is $K(G,1)$ space. Thus, $\mathrm{cd}(G^r)=3r$. Moreover, note that $\mathrm{cd}(Z)\geq 1$.
This gives us $\mathrm{cd}(H)\leq 3r-1$. Moreover, from \cite{EilenbergGanea}, we have $\cat(H)=\mathrm{cd}(H)+1$, implying the inequality $\TC_r(M_O)\leq \cat(H)=\mathrm{cd}(H)+1\leq 3r$.
\end{proof}

At this stage we note some important observations.

\begin{remark}\label{remk:heisen}\
\begin{enumerate}
    \item It is known that the center of $\pi_1(M_N)$  is trivial (see \cite[Section 2.3.1]{Surveysf}). Therefore, we cannot use \Cref{prop:tcnub} to improve the dimensional upper bound on $\TC_r(M_N)$.  

    \item It follows from \cite[Theorem 2]{teichner1996maximalnilpotentquotients3manifold} that the fundamental group of an aspherical oriented Seifert fibered manifold is nilpotent if  and only if it is  a Heisenberg group (see also \cite[page 9] {lee2001seifertmanifolds}). Moreover, for a torsion free nilpotent group $G$, the quotient group $G^r/d_r(Z)$ has finite cohomological dimension (see Proposition \ref{lemma: nilpotent torsion free}). Thus, for a Seifert manifold,  $M=K(G,1)$, with $G$ a Heisenberg group we have $\TC_r(M)\leq 3r$ using \Cref{lem:tck-MO-ub}.
 \end{enumerate}
\end{remark}

Next, we provide our first estimation of the higher topological complexity of the aspherical Seifert fibred manifolds.

\begin{theorem}\label{thm: TCofSeifert}
Suppose 
\[
M_O = (O,o,g \mid e : (a_1,b_1),\dots,(a_m,b_m))
\quad \text{and} \quad
M_N = (O,n,g \mid e : (a_1,b_1),\dots,(a_m,b_m)),
\]
where $g \ge 1$ and $n_2 > 2$. Then
\[
\TC_r(M_O), \TC_r(M_N) \in \{3r,3r+1\},
\]
for $r = 2$ whenever the following conditions hold;
\begin{enumerate}
    \item $a_1\equiv {2(\rm mod}~ 4)$ or
\item $a_k\equiv {2(\rm mod}~ 4)$ and $a_j\equiv {2(\rm mod}~ 4)$ for some $2\leq j\neq k\leq n_2$,
\end{enumerate}
and for $r \ge 3$ whenever $a_k\equiv 2(\rm {mod}~ 4)$ and $a_j\equiv 2(\rm {mod}~ 4)$ for some $1\leq j\neq k\leq n_2$.
\end{theorem}
\begin{proof}
We establish the result for $M_O$; similar arguments apply to $M_N$ as well.
Let $x=\alpha_j$ and $y=\alpha_k$ be in $H^1(M;\Z_2)$ and $p_i:M^r\to M$ be a projection onto the $i^{\text{th}}$ factor. Consider $\bar{x}_i=p_i^{\ast}(x)-p_1^{\ast}(x)$. Then note that $d_r^{\ast}(\bar{x}_i)=0$, where $d_r:M\to M^r$ is the diagonal map. Observe that the product $\prod_{i=2}^r\bar{x}_i\neq 0$ since  the expression contains a non-zero term $\prod_{i=2}^rp_i^{\ast}(x)=1\otimes \alpha_2\otimes \dots\otimes \alpha_2$. 

Let $\bar{y}_i=p_i^{\ast}(y)-p_1^{\ast}(y)$ and $B_2$ be the mod-$2$ Bockstein. Recall from  \Cref{remk:main123} $B_2(y)=\alpha_3^2$. Note that $B_2(\bar{y}_i)=B_2(p_i^{\ast}(y))-B_2(p_1^{\ast}(y))=p_i^{\ast}(B_2(y))-p_1^{\ast}(B_2(y))=\overline{B_2(y)}_i$ for $2\leq i\leq r$. 

Now consider the product $\prod_{i=2}^r\overline{B_2(y)}_i$. We  observe that the product $\prod_{i=2}^r\overline{B_2(y)}_i$ is non-zero as it contains the term $1\otimes \alpha_3^2\otimes \dots\otimes \alpha_3^2$. Then one can show the following equality by induction.
\vspace{-3mm}
\begin{align*}\label{eqn:66}
\overline{B_2(x)}_2\cdot \prod_{i=2}^r\bar{x}_i\cdot \prod_{i=2}^r\overline{B_2(y)}_i & =\binom{a_1}{2}^{r-1}\alpha_j^2\otimes \gamma \otimes\dots \otimes\gamma\\ &+\binom{a_1}{2}^{r-1}\gamma\otimes\alpha_j^2\otimes\gamma\otimes\dots \otimes\gamma \\&+\binom{a_1}{2}^{r-2}\bigg[\binom{a_1}{2}
+\binom{a_2}{2}\bigg]\gamma\otimes\alpha_k^2\otimes\gamma\otimes\dots \otimes \gamma\\&+ \binom{a_1}{2}^{r-2}\bigg[\binom{a_1}{2}+\binom{a_2}{2}\bigg]\alpha_k^2\otimes\gamma\otimes\dots \otimes \gamma. 
\end{align*}

To write the above expression in a simple form we introduce the notations: 
$$A_1=\alpha_j^2\otimes \gamma \otimes\dots \otimes\gamma,  \quad
A_2=\gamma\otimes\alpha_j^2\otimes\gamma\otimes\dots \otimes\gamma,
\quad A_3=\gamma\otimes\alpha_k^2\otimes\gamma\otimes\dots \otimes \gamma, $$
and $A_4= \alpha_k^2\otimes\gamma\otimes\dots \otimes \gamma$.
In these above notations, we rewrite $\overline{B_2(x)}_2 \cdot \prod_{i=2}^r \bar{x}_i \cdot \prod_{i=2}^r \overline{B_2(y)}_i$ as follows:
\begin{align*}
\overline{B_2(x)}_2 \cdot \prod_{i=2}^r \bar{x}_i \cdot \prod_{i=2}^r \overline{B_2(y)}_i 
&= \binom{a_1}{2}^{\, r-2} \left\{
\binom{a_1}{2}A_1
+ \binom{a_1}{2}A_2 \right. \\
&\quad \left.
+ \left[\binom{a_1}{2} + \binom{a_j}{2}\right]A_3
+ \left[\binom{a_1}{2} + \binom{a_j}{2}\right]A_4.
\right\}
\end{align*}
Observe that $\overline{B_2(x)}_2 \cdot \prod_{i=2}^r \bar{x}_i \cdot \prod_{i=2}^r \overline{B_2(y)}_i \neq 0$ if and only if 

$$\binom{a_1}{2}^{r-2}\neq 0 \text{ and } \binom{a_1}{2}A_1+\binom{a_1}{2}A_2 + \bigg[\binom{a_1}{2}
 +\binom{a_2}{2}\bigg]A_3+ \bigg[\binom{a_1}{2}
+\binom{a_2}{2}\bigg]A_4\neq 0.$$

Observe that from 
 \Cref{thm: cring1} we have the identity
\[\alpha_k^2=\binom{a_1}{2}\sum_{j=2}^{n_2}\beta_j+\binom{a_k}{2}\beta_k= \binom{a_1}{2}\sum_{j=2, j\neq 3}^{n_2}\beta_j+\bigg[\binom{a_1}{2}+\binom{a_k}{2}\bigg]\beta_k.\]
Note that $\alpha_k^2=0$ if and only if $\binom{a_1}{2}=0$ and $\binom{a_k}{2}=0$. 

Thus, by symmetry $\overline{B_2(x)}_2 \cdot \prod_{i=2}^r \bar{x}_i \cdot \prod_{i=2}^r \overline{B_2(y)}_i =0$ if and only if   $\binom{a_1}{2}^{r-2}= 0$ or 
\begin{enumerate}
\item[(i)] 
$\alpha_k^2=0$ or $\left[\binom{a_1}{2}+\binom{a_j}{2}\right]=0$ and 
\item[(ii)] $\alpha_j^2=0$ or $\binom{a_1}{2}=0$.
\end{enumerate}
From this we conclude that, for $r=2$, $\overline{B_2(x)}_2 \cdot \prod_{i=2}^r \bar{x}_i \cdot \prod_{i=2}^r \overline{B_2(y)}_i\neq 0$ if and only if  
\begin{enumerate}
    \item $a_1\equiv {2(\rm mod}~ 4)$ or
\item $a_k\equiv {2(\rm mod}~ 4)$ and $a_j\equiv {2(\rm mod}~ 4)$.
\end{enumerate}

and for $r\geq 3$ $\overline{B_2(x)}_2 \cdot \prod_{i=2}^r \bar{x}_i \cdot \prod_{i=2}^r \overline{B_2(y)}_i\neq 0$ if and only if $a_1\equiv {2(\rm mod}~ 4)$ and
\begin{enumerate}
    \item $a_1\equiv {2(\rm mod}~ 4)$ or
\item $a_k\equiv {2(\rm mod}~ 4)$ and $a_j\equiv {2(\rm mod}~ 4)$.
\end{enumerate}
The condition for $r\geq 3$ is equivalent to  $a_k\equiv {2(\rm mod}~ 4)$ and $a_j\equiv {2(\rm mod}~ 4)$ for some $1\leq j\neq k\leq n_2$.

Therefore, we know that $\overline{B_2(x)}_2 \cdot \prod_{i=2}^r \bar{x}_i \cdot \prod_{i=2}^r \overline{B_2(y)}_i\neq 0$ if and only if the hypotheses of the theorem satisfies. We now use the cohomological lower bound to conclude our result.
Note that \Cref{thm: c wt2} gives $\wgt_{e_r}(\overline{B_2(y)}_i)\geq 2$ and $\wgt_{e_r}(\overline{B_2(x)}_2)\geq 2$. Therefore, using \cite[Proposition 2]{GrantCweights} we get that
\vspace{-3mm}
\[\TC_r(M_O)>\wgt_{e_r}\bigg(\overline{B_2(x)}_2\cdot \prod_{i=2}^r\bar{x}_i\cdot \prod_{i=2}^r\overline{B_2(y)}_i\bigg)\geq 3r-1.\]
By performing the similar calculations we obtain $\TC_r(M_N)\geq 3r$.
The inequalities $$\TC_r(M_O),\TC_r(M_N)\leq 3r+1$$ follows from \cite[Theorem 3.9]{RUDIY2014}.
\end{proof}

\begin{remark}
Since $\bar{B}_2(\alpha_i)=\bar{\alpha}_i^2$, the bound can be obtained using cohomology weights and zero divisor cup length are the same.  We use the Bockstein notation in \Cref{thm: TCofSeifert} for comparison with the proof of the next theorem. 
\end{remark}

\begin{theorem}\label{prop:n2equals1higherTC}
Suppose $M_O= (O,o,g \mid e: (a_1,b_1),\dots, (a_m,b_m))$ and $M_N=(O,n,g \mid e: (a_1,b_1),\dots, (a_m,b_m))$ with
$g\geq 1$.
If $n_2= 1 \text{ or } 2$ and $n_p >2$ for some  odd prime $p$, $a_1'c_1\not\equiv 0 (\rm {mod } ~p)$ and $a_j'c_j+a_1'c_1 \not\equiv 0 (\rm {mod } ~p) $ for some  $j \leq n_p$.
Then $$\TC_r(M_O), \TC_r(M_N)\in\{3r, 3r+1\}.$$
\end{theorem}
\begin{proof}

Let $n_2=1$ or $2$ and there exist an odd prime $p$ such that $n_p>2$. 
Let $p$ be an odd prime. Let $k\leq n_p$ and $y_i=p_i^{\ast}(B_p(\alpha_k))$ for $1\leq i\leq r$, where $B_p$ is mod-$p$ Bockstein. 
Consider the higher zero divisors $\bar{y}_i=y_i-y_1$ for $2\leq i\leq r$ and their product 
$\prod_{i=2}^{r}\bar{y}_i$. Since $y_i$'s are of degree-$2$ cohomology classes, their squares are zero for degree reasons.
Therefore, we obtain the following expression:
\[\prod_{i=2}^{r}\bar{y}_i= y_2\cdots y_r- \sum_{i=2}^ry_1\cdots y_{i-1}\hat{y_i}y_{i+1}\cdots y_r,\] where $\hat{y_i}$ denotes that $y_i$ is missing in the product.

Let $j\leq n_p$ such that $j\neq k$ and $z_i=p_i^{\ast}(B_p(\alpha_j))$ for $1\leq i\leq r$.
We consider another higher zero divisor $\bar{z}_2=z_2-z_1$ and its product with $\prod_{i=2}^{r}\bar{y}_i$ as follows:
\[\bar{z}_2\cdot \prod_{i=2}^{r}\bar{y}_i=-y_1z_2y_3\cdots y_r -z_1y_2\cdots y_r\] because of the degree reasons. 
With same $j$ as before, we take $x_i=p_i^{\ast}(\alpha_j)$ for $1\leq i\leq r$ and consider 
 another set of higher zero divisors $\bar{x}_i=x_i-x_1$. 
Note that we have the generalized product rule 
\begin{equation}\label{eq: genprodform}
(a_1\otimes \cdots \otimes a_r)\cdot (b_1\otimes \cdots \otimes b_r)=(-1)^{\sum_{j=0}^{r-2} s_{r-j}\sum_{i=1}^{r-j-1}t_i }(a_1b_1\otimes \cdots \otimes a_rb_r),    
\end{equation}
where $|a_i|=s_i$ and $|b_i|=t_i$ for $1\leq i,j\leq r$.
Now consider the product $\prod_{i=2}^r\bar{x}_i$.
We use \eqref{eq: genprodform} to obtain the following
\[\prod_{i=2}^r\bar{x}_i=x_2\cdots x_r+ \sum_{i=2}^r (\pm)x_{1}\cdots \hat{x}_{i}\cdots x_{r} +P,\]
where $P$ is the sum of products containing square terms. Note that $\prod_{i=2}^r\bar{x}_i\neq 0$ as it contains the unique non-zero term $x_2\cdots x_r$.
Since we are going to multiply $\prod_{i=2}^r\bar{x}_i$ with $\bar{z}_2\cdot \prod_{i=2}^{r}\bar{y}_i$ and all $y_i$'s and $z_i$'s are of degree-$2$, the product of terms in $P$ with terms in $\bar{z}_2\cdot \prod_{i=2}^{r}\bar{y}_i$ becomes zero for degree reason.
Therefore, we have the product 
\begin{equation}\label{eq: zcls}
\prod_{i=2}^r\bar{x}_i\cdot \bar{z}_2\cdot \prod_{i=2}^{r}\bar{y}_i= -y_1z_2x_2y_3x_3\cdots y_rx_r-z_1y_2x_2\cdots y_rx_r+ Q,    
\end{equation}
where 
$Q=(-y_1z_2y_3\cdots y_r-z_1y_2\cdots y_r)\cdot  \sum_{i=2}^r (\pm)x_{1}\cdots \hat{x}_{i}\cdots x_{r}.$
Note that $Q$ contains terms $c_1\otimes \cdots \otimes  c_r$ such that  $|c_1|=3$. 
Therefore, the terms  \[R=-y_1z_2x_2y_3x_3\cdots y_rx_r-z_1y_2x_2\cdots y_rx_r\] in \eqref{eq: zcls} cannot get cancelled by the terms in $Q$.
Now after putting the explicit values of $y_i$'s and $z_i$'s and using relations in \Cref{thm: cring1}, we get the following
\[R= - \bigg[(a_j'c_j+E)E^{r-2}(-a_k'c_k\beta_k+E\beta_1)\otimes \gamma\otimes \cdots \otimes \gamma + E^{r-1}(-a_j'c_j\beta_j+E\beta_1)\otimes \gamma\otimes \cdots \otimes \gamma\bigg],\]\
where $E=a_1'c_1$.
One can observe that $R$ is nonzero if and only if $a_1'c_1\not\equiv 0 (\rm {mod } ~p)$ and $a_j'c_j+a_1'c_1 \not\equiv 0 (\rm {mod } ~p) $ for some  $j \leq n_p$.
Consequently, the product in \eqref{eq: zcls} is nonzero if the hypothesis of the theorem is satisfied.
Note that the $\wgt_{e_r}(\bar{z}_2)=\wgt_{e_r}(\bar{y}_i)=2$. Therefore, we get the lower bound  $\TC_r(M_O)\geq 3r$. 
The similar calculations can be done to prove the inequality $\TC_r(M_N)\geq 3r$. 
The inequalities  $\TC_r(M_O),\TC_r(M_N)\leq 3r+1$ follows from \cite[Theorem 3.9]{RUDIY2014}.
\end{proof}

\begin{corollary}\label{cor: equality}
    If $M_O$ satisfies the hypothesis of \Cref{thm: TCofSeifert} or \Cref{prop:n2equals1higherTC} and $G=\pi_1(M_O)$ with $\mathrm{cd}(G^r/d_r(Z))$ is finite, where $Z$ is the center of $G$. Then $$\TC_r(M_O)=3r.$$
\end{corollary}

\begin{theorem}\label{prop:n2equals0higherTC} 
Suppose $M_O= (O,o,g \mid e: (a_1,b_1),\dots, (a_m,b_m))$ and $M_N=(O,n,g \mid e: (a_1,b_1),\dots, (a_m,b_m))$ with $n_2=0$ and $(a_i,a_j)=1$ for all $1\leq i< j\leq m$ and $Ae+C\equiv 0(\rm{mod} ~ 2)$.
Let $s$ be a positive integer such that $b_i\equiv 0({\rm mod}~2)$ for $1\leq i\leq s$ and $b_i\not\equiv 0({\rm mod}~2)$ for $s+1\leq i\leq m$.
Suppose $\frac{1}{A}(\sum_{i=1}^{s}b_i'A_i+\frac{Ae+C}{2})$ is non-zero. 
Then $$3r-1\leq \TC_r(M_O), \TC_r(M_N)\leq 3r+1.$$
\end{theorem}
\begin{proof}
We have assumed $Ae+C\equiv 0({\rm mod}~2)$.  Therefore, by \Cref{thm: cring2},  $\alpha\in H^1(M;\Z_2)$.
Consider the higher zero divisors $\bar{x}_i=p_i^{\ast}(\alpha)-p_1^{\ast}(\alpha)$.  Observe that the product $\prod_{i=2}^r\bar{x}_i\neq 0$. Let $\lambda=\frac{1}{A}(\sum_{i=1}^{s}b_i'A_i+\frac{Ae+C}{2})$.
From \Cref{thm: cring2}, we have $B_2(\alpha)=\lambda\beta$.
For $1\leq l\leq g$, consider $\bar{\theta}_l=p_2^{\ast}(\theta_l)=p_1^{\ast}(\theta_l)$. 
One can observe that the product   $\prod_{i=2}^{r}\bar{x}_i\cdot \prod_{j=2}^{r}\bar{B}_2(\alpha)_j\cdot \bar{\theta}_l$ is non-zero as it contains the term $\lambda^{r-1}\theta_l\otimes\gamma\otimes\dots\otimes\gamma$ if $\lambda$ is non-zero. 
This gives 
\vspace{-3mm}
\[\TC_r(M_O)>\wgt_{e_r}\bigg(\overline{B_2(x)}_2\cdot \prod_{i=2}^r\bar{x}_i\cdot \prod_{i=2}^r\overline{B_2(y)}_i\bigg)\geq 3r-2\] using \cite[Proposition 2]{GrantCweights}.
Similar calculations can be done to show $\TC_r(M_N)\geq 3r-1$.
The inequalities  $\TC_r(M_O),\TC_r(M_N)\leq 3r+1$ follows from \cite[Theorem 3.9]{RUDIY2014}.
\end{proof}

\begin{corollary}
    If $M_O$ satisfies the hypothesis of \Cref{prop:n2equals0higherTC} and $G=\pi_1(M_O)$ with $\mathrm{cd}(G^r/d_r(Z))$ is finite, where  $Z$ is the centre of $G$. Then $\TC_r(M_O)\in\{3r-1,3r\}$. 
\end{corollary}
   
\begin{remark}
Observe that, in \Cref{prop:n2equals0higherTC} we have used mod-$2$ cohomology ring description of $M_O$ and $M_N$ to compute the cohomological lower bound on the $\TC_r(M_O)$ and $\TC_r(M_N)$.
One can check that the mod-$p$ cohomology ring description also gives the same cohomological lower bound on both $\TC_r(M_O)$ and $\TC_r(M_N)$.
\end{remark}

\begin{example}\label{example}
Consider $S^1\to M\to \Sigma_g$ be a $S^1$-bundle over the orientable surface of genus $g>0$ with $M$ orientable and $e>0$ as the Euler number. It is known that $M$ is aspherical  Seifert-fibered manifold of type $(O,o,g \mid e: (1,e))$. 
Then it follows from \Cref{prop:n2equals0higherTC} that the $3r-1\leq\TC_r(M)\leq 3r+1$.   
On the other hand, if $e=0$, then $M\cong \Sigma_g\times S^1$. In this case, we have $\TC_r(M)=2r$ if $g=1$ and $\TC_r(M)=2r+2$ otherwise. 
We note that $e=0$ case cannot be handled using \Cref{prop:n2equals0higherTC}.
Daundkar  considered the case $g=0$ in \cite[Corollary 5.8]{daundkar2023group}, that is $M$ is a lens space. In particular, it was shown that $\TC_r(M)=3r$.
\end{example}

\section{Wedge and connected sum of some 3-manifolds and concluding remarks}
In general, the problem of computing the higher topological complexity of arbitrary
$3$--manifolds remains widely open. Although a complete formula appears out of reach,
it is nevertheless possible to obtain sharp bounds, and in some cases exact values,
for specific classes of $3$--manifolds. In particular, since there is no known formula
expressing the higher topological complexity of a connected sum in terms of that of
its summands, a comprehensive classification seems unlikely at present.

In this section, we focus on certain classes of $3$--manifolds that are not
Seifert--fibered and obtain explicit computations and estimates for their higher
topological complexity. These results complement the main theorems of the paper and
highlight directions for further investigation.

\begin{proposition}\label{prop: cat conn sum}
Let $X=\#_{k}(S^2\times S^1)$. Then $
\cat(X)=3$.
\end{proposition}
\vspace{-3mm}
\begin{proof}
The proof follows from \cite[Proposition 11]{catconnsum} and the fact that $\cat(S^2\times S^1)=3$.
\end{proof}

\begin{theorem}
Let $X=\#_{k}(S^1\times S^2)$. Then $\TC_r(X)=2r+1$.
\end{theorem}
\vspace{-3mm}
\begin{proof}
Here we prove the theorem for $k=2$. A similar argument works for the general case.
Recall that the integral cohomology ring of $X$ is given as follows: 
\vspace{-2mm}
\[H^{\ast}(X;\Z)\cong \frac{ \Lambda[x_1,x_2]\oplus\Lambda[y_1,y_2]}{\left<x_1x_2+x_2x_1, y_1y_2+y_2y_1, x_1x_2-y_1y_2, x_ix_j, y_jx_i, \right>},\]
where $\Lambda$ is the exterior product over integers and $|x_i|=i=|y_i|$ for $i=1,2$. Let $u=x_1$, $v=x_2$, $w=y_1$ and $z=y_2$.
Let $\bar{u}_i=p_i^{\ast}(x_1)-p_1^{\ast}(x_1)$ , $\bar{v}_i=p_i^{\ast}(x_2)-p_1^{\ast}(x_2)$, $\bar{w}_i=p_i^{\ast}(y_1)-p_1^{\ast}(y_1)$ and $\bar{z}_i=p_i^{\ast}(y_2)-p_1^{\ast}(y_2)$. Then $\bar{u}_i$, $\bar{v}_i$ , $\bar{w}_i$ and $\bar{z}_i$ are in $\ker(d_r^{\ast})$. 
Note that the products 
$\prod_{i=2}^{r}\bar{u}_i$ is  non-zero as it contain the terms $1\otimes x_1\otimes\dots\otimes x_1$. Similarly, $\prod_{i=2}^{r}\bar{v}_i$ is non-zero.
We now consider the product $\bar{w}_2\cdot\bar{z}_2\cdot \prod_{i=2}^{r}\bar{u}_i\cdot \prod_{i=2}^{r}\bar{v}_i$, which is non-zero as it contains the term $y_1y_2\otimes x_1x_2\otimes \dots\otimes x_1x_2$ and $x_1x_2$ generates $H^3(S^2\times S^1)$. 

By \cite[Proposition 3.4]{RUD2010}, we get that $2r+1\leq\TC_r(X)$. We then observe that $\TC_r(X)\leq \cat(X^r)\leq 2r+1$ as $\cat(X)=3$. This concludes the result.
\end{proof}

The topological complexity of wedges and connected sums of spaces has been studied by Dranishnikov in \cite{Dr3} and jointly with Sadykov in \cite{catconnsum, Drani}. More recently, Neofytidis investigated this problem for aspherical manifolds.
However, the relationship between $\TC(M_1\# M_2)$ and $\TC(M_1\vee M_2)$ is not
fully understood. In particular, it is natural to ask whether the inequality
\[
\TC(M_1\# M_2)\leq \TC(M_1\vee M_2)
\]
holds for a given pair of manifolds. In \cite[Corollary~5.2]{neofytidis2022},
Neofytidis showed that the topological complexity of the connected sum of a negatively
curved $4$--manifold with nonzero second Betti number and any aspherical
$4$--manifold is equal to $9$.

In contrast, we show below that the topological complexity of the wedge of closed,
orientable, aspherical $3$--manifolds (including, for instance, the Seifert manifolds
considered earlier) is equal to $7$. As a consequence, for these manifolds the
inequality
\[
\TC(M_1\# M_2)\leq \TC(M_1\vee M_2)
\]
does indeed hold. Our proof follows an argument similar in spirit to that of
\cite[Corollary~5.2]{neofytidis2022}.

\begin{proposition}
Let $M_i$ be closed, orientable, aspherical $3$--manifolds for $1 \le i \le k$. Then
\begin{equation}\label{eq: tc of wedge}
\TC(M_1 \vee \cdots \vee M_k) = 7.    
\end{equation}

Moreover, if one of the $M_i$ is negatively curved with $H_2(M_i,\mathbb{Q})\neq 0$, then 
\begin{equation}\label{eq: tc of connected sum}
6\leq \TC(M_1\#\dots \# M_k)\leq 7.    
\end{equation}

\end{proposition}

\begin{proof}
Let $B\pi_1(M_1 \# \cdots \# M_k)$ denote the classifying space of
$\pi_1(M_1 \# \cdots \# M_k)$. Since each $M_i$ is aspherical, we have
\[
\begin{aligned}
\TC(M_1 \vee \cdots \vee M_k)
&= \TC\big(B\pi_1(M_1 \# \cdots \# M_k)\big) \\
&= \TC\big(\pi_1(M_1) \ast \cdots \ast \pi_1(M_k)\big).
\end{aligned}
\]

Let $G_i = \pi_1(M_i)$. Using \cite[Theorem~2]{Drani} iteratively, we obtain
\[
\TC\big(\pi_1(M_1) \ast \cdots \ast \pi_1(M_k)\big)
=
\max_{\substack{1 \le i \le k \\ 1 \le s \le k-1}}
\left\{
\TC(\pi_1(M_i)),\;
\mathrm{cd}\big((G_1 \ast \cdots \ast G_s)\times G_{s+1}\big) +1
\right\}.
\]
Note that
$
\mathrm{cd}\big((G_1 \ast \cdots \ast G_s)\times G_{s+1}\big) = 6
\quad \text{for all } 1 \le s \le k-1,
$
since
\[
K\big((G_1 \ast \cdots \ast G_s)\times G_{s+1},1\big)
\cong
(M_1 \# \cdots \# M_s) \times M_{s+1}
\]
is an orientable $6$-manifold.
Therefore, $\TC(M_1 \vee \cdots \vee M_k) = 7.$

The other inequality $6\leq \TC(M_1\#\dots \# M_k)\leq 7$ follows from \cite[Corollary 1.2]{neofytidis2022}.
\end{proof}

We now extend \eqref{eq: tc of wedge} to the higher setting. To this end, we first establish several auxiliary results. 
\begin{lemma}\label{lem: retract}
Let $A\subseteq X$ be a retract. Then $\TC_r(A)\leq \TC_r(X)$.
\end{lemma}
The proof the above lemma is a straightforward analogue of  \cite[Lemma~4.25]{F-book} for higher TC. 

\begin{proposition}\label{prop: higher tc wedge}
For any path connected topological spaces $X$ and $Y$, we have

\begin{enumerate}
    \item if $r$ is even, $$\TC_r(X\vee Y)\geq \mathrm{max}\{\TC_r(X), \TC_r(Y), \cat((X\times Y)^{\lfloor \frac{r}{2}\rfloor})\},$$ and, 
    \item if $r$ is odd, $$\TC_r(X\vee Y)\geq \mathrm{max}\{\TC_r(X), \TC_r(Y), \cat((X\times Y)^{\lfloor \frac{r}{2}\rfloor}\times X), \cat((X\times Y)^{\lfloor \frac{r}{2}\rfloor}\times Y)\}.$$
\end{enumerate}

\end{proposition}
\begin{proof}
The proof follows arguments analogous to those in \cite[Theorem 3.6]{Dr3}.
Since both $X$ and $Y$ are retracts of $X\vee Y$, 
the inequalities $\TC_r(X), \TC_r(Y)\leq \TC(X\vee Y)$ follow from Lemma \ref{lem: retract}.

First, consider the case $r=2k$. We need to show the inequality $\cat((X\times Y)^{\lfloor \frac{r}{2}\rfloor})\leq \TC_r(X\vee Y)$.
Note that, we can identify $(X\times Y)^k$ as a subspace of $(X\vee Y)^r$ and assume that it is covered by $\TC_r(X\vee Y)$-many open sets $U$ such that all projections $pr_i: U\to X\vee Y$ for $1\leq i\leq r$, are homotopic to each other. Therefore, we can choose a homotopy $H_U:U\times I\to X\vee Y$ such that $H_U((x_1,y_1),\dots, (x_k,y_k),\frac{2i-2}{r-1})=x_i$ and $H_U((x_1,y_1),\dots, (x_k,y_k),\frac{2i-1}{r-1})=y_i$ for $1\leq i\leq k$.
Denote $\widetilde{xy} =((x_1,y_1),\dots, (x_k,y_k))$. Suppose $r_X:X\vee Y\to X$ and $r_Y:X\vee Y\to Y$ are the retraction maps. 
Then define $G_U:U\times I\to (X\times Y)^k$ by 
\[
G_U(\widetilde{xy}, t)
:=
\bigl(
r_X H_U\!\left(\widetilde{xy}, \tfrac{t}{r-1}\right),
r_Y H_U\!\left(\widetilde{xy}, \tfrac{1-t}{r-1}\right),
\ldots,
r_X H_U\!\left(\widetilde{xy}, \tfrac{t+r-2}{r-1}\right),
r_Y H_U\!\left(\widetilde{xy}, \tfrac{r-1-t}{r-1}\right)
\bigr).
\]
Observe that $G(\widetilde{xy},0)=\widetilde{xy}$ and $G_U(\widetilde{xy},1)=(v_0,\dots, v_0)$, where $v_0$ is the wedge point. This implies $U$ is a contractible subspace of $(X\times Y)^k$. 

We now consider the case $r=2k+1$. We show the inequality $\cat((X\times Y)^{\lfloor \frac{r}{2}\rfloor})\times X\leq \TC_r(X\vee Y)$. 

Similarly, as in the first case, we can consider $(X\times Y)^k\times X\subseteq (X\vee Y)^r$ and assume that it is covered by $\TC_r(X\vee Y)$-many open sets $U$ such that all projections $pr_i: U\to X\vee Y$ for $1\leq i\leq r$ are homotopic to each other.
We can choose a homotopy $H_U:U\times I\to X\vee Y$ such that $H_U((x_1,y_1),\dots, (x_k,y_k),\frac{2i-2}{r-1})=x_i$ for  $1\leq i\leq k+1$ and $H_U((x_1,y_1),\dots, (x_k,y_k),\frac{2i-1}{r-1})=y_i$ for $1\leq i\leq k$.
Denote $\widetilde{xy} =((x_1,y_1),\dots, (x_k,y_k),x_{k+1})$.
Define the homotopy $G_U: U\times I\to (X\times Y)^k\times X$ by
\[
G_U(\widetilde{xy}, t)
:=
\bigl(
r_X H_U\!\left(\widetilde{xy}, \tfrac{t}{r-1}\right),
r_Y H_U\!\left(\widetilde{xy}, \tfrac{1-t}{r-1}\right),
\ldots,
r_Y H_U\!\left(\widetilde{xy}, \tfrac{t+r-3}{r-1}\right),
r_X H_U\!\left(\widetilde{xy}, \tfrac{r-1-t}{r-1}\right)
\bigr).
\]
This homotopy contracts $U$ to a point in $(X\times Y)^k\times X$. This concludes the proof. 
\end{proof}

In general, for path-connected spaces $X$ and $Y$, the inequality $$\cat(X),\cat(Y)\leq \cat(X\times Y)$$ holds. As a consequence, we have the following corollary.
\begin{corollary}
The following inequality holds for any $r$ 
\[\TC_r(X\vee Y)\geq  \mathrm{max}\{\TC_r(X), \TC_r(Y), \cat((X\times Y)^{\lfloor \frac{r}{2}\rfloor})\}.\]     
\end{corollary}

By iteratively applying Proposition~\ref{prop: higher tc wedge}, one obtains the following straightforward extension.
\begin{proposition}\label{prop: iterated wedge sum}
 Let $X_i$ be path connected topological spaces for $1\leq i\leq k$ and $Z=X_1\vee\dots \vee X_k$. Then if $r$ is even, we have
\[
\TC_r(Z)
\geq 
\max_{\substack{1 \le i \le k \\ 1 \le j \le k-1}}
\left\{
\TC_r(X_i),\;
\cat\big((X_j\times \widetilde{X}_j)^{\lfloor \frac{r}{2}\rfloor}\big)
\right\},
\] 
  and, if $r$ is odd, we have
\[
\TC_r(Z)
\geq 
\max_{\substack{1 \le i \le k \\ 1 \le j \le k-1}}
\left\{
\TC_r(X_i),\;
\cat\big((X_j\times\Tilde{X}_j)^{\lfloor \frac{r}{2}\rfloor} \times X_j\big), \cat\big((X_j\times \widetilde{X}_j)^{\lfloor \frac{r}{2}\rfloor}\times \widetilde{X}_j\big)
\right\},
\] 
 where $\widetilde{X}_j=(X_{j+1}\vee\dots\vee X_k)$.
\end{proposition}

\begin{proposition}\label{prop: higher tc of iterated wedge}
Let $M_i$ be closed, orientable, aspherical $3$--manifolds for $1 \le i \le k$. Then
\[
\TC_r(M_1 \vee \cdots \vee M_k) = 3r+1.
\]    
\end{proposition}

\begin{proof}
The upper bound $\TC_r(M_1 \vee \cdots \vee M_k)\leq 3r+1$ follows from the dimensional upper bound on the higher topological complexity.

Since $M_i$'s are aspherical oriented $3$-manifolds, $(M_{k-1}\times M_k)^{\lfloor \frac{r}{2}\rfloor}$ and $(M_{k-1}\times M_k)^{\lfloor \frac{r}{2}\rfloor} \times M_{k-1}$ are also aspherical, oriented manifolds of dimension $6\lfloor \frac{r}{2}\rfloor $ and $6\lfloor \frac{r}{2}\rfloor+3$, respectively. Therefore, their cohomological dimensions are the same as their usual dimensions i.e., 
$$\mathrm{cd}((M_{k-1}\times M_k)^{\lfloor \frac{r}{2}\rfloor})=6\bigg\lfloor \frac{r}{2}\bigg\rfloor$$ and $$\mathrm{cd}((M_{k-1}\times M_k)^{\lfloor \frac{r}{2}\rfloor} \times M_{k-1})=6\bigg\lfloor \frac{r}{2}\bigg\rfloor+3.$$
Then, it follows from \cite{EilenbergGanea} that $\cat((M_{k-1}\times M_k)^{\lfloor \frac{r}{2}\rfloor})=3r+1$ if $r$ is even and  $\cat((M_{k-1}\times M_k)^{\lfloor \frac{r}{2}\rfloor} \times M_{k-1})=3r+1$ if $r$ is odd.
Finally, the lower bound $\TC_r(M_1\vee\dots\vee M_k)\geq 3r+1$ follows from Proposition  \ref{prop: iterated wedge sum}.
\end{proof}

\vspace{.5cm}

\subsection*{Further directions}
It seems reasonable to expect that  the  higher analogue of the inequality
\eqref{eq: tc of connected sum} should be true. Pursuing this direction would likely require
extending the results of Neofytidis \cite{neofytidis2022} to the higher setting, 
which would broaden the scope of the present work. 

Another question that would be interesting to explore arises in the case of Proposition~\ref{prop: higher tc wedge}. Although our arguments  do not require  a sharp lower bound, it will be useful to know whether the bound is sharp in general. 
This  would likely require techniques different
from those developed by Dranishnikov \cite{Drani}.

\vspace{0.5cm}

\noindent\textbf{Acknowledgment}: We thank the anonymous referee for their valuable comments and suggestions, which have significantly improved the quality and scope of the paper. We are also grateful to Mark Grant and Christoforos Neofytidis for sharing his insights on our work.

Navnath Daundkar acknowledges the support of the Indian Institute of Technology Bombay, where a major part of this work was carried out during his postdoctoral position. He also gratefully acknowledges the support of the National Board for Higher Mathematics (NBHM) through Grant No.~0204/10/(16)/2023/R\&D-II/2789 and the DST--INSPIRE Faculty Fellowship (Faculty Registration No.~IFA24-MA218), Department of Science and Technology, Government of India.
A substantial portion of this project, carried out by S.~Thandar, was undertaken during his time at IIT Bombay, where he was supported by a Senior Research Fellowship from the University Grants Commission (UGC), India. The remaining work was completed during his tenure as a Visiting Fellow at the Tata Institute of Fundamental Research (TIFR).

\bibliographystyle{plain} 
\bibliography{references}

@book {F-book,
    AUTHOR = {Farber, Michael},
     TITLE = {Invitation to topological robotics},
    SERIES = {Zurich Lectures in Advanced Mathematics},
 PUBLISHER = {European Mathematical Society (EMS), Z\"urich},
      YEAR = {2008},
     PAGES = {x+133},
      ISBN = {978-3-03719-054-8},
   MRCLASS = {55R80 (55M30 57R70 58E05 68T40)},
  MRNUMBER = {2455573},
MRREVIEWER = {Yuli\ B.\ Rudyak},
       DOI = {10.4171/054},
       URL = {https://doi.org/10.4171/054},
}

@article {Dr3,
    AUTHOR = {Dranishnikov, Alexander},
     TITLE = {Topological complexity of wedges and covering maps},
   JOURNAL = {Proc. Amer. Math. Soc.},
  FJOURNAL = {Proceedings of the American Mathematical Society},
    VOLUME = {142},
      YEAR = {2014},
    NUMBER = {12},
     PAGES = {4365--4376},
      ISSN = {0002-9939,1088-6826},
   MRCLASS = {55M30 (54F45 57N65)},
  MRNUMBER = {3267004},
MRREVIEWER = {Michael\ S.\ Farber},
       DOI = {10.1090/S0002-9939-2014-12146-0},
       URL = {https://doi.org/10.1090/S0002-9939-2014-12146-0},
}

@misc{teichner1996maximalnilpotentquotients3manifold,
      title={Maximal Nilpotent Quotients of 3-Manifold Groups}, 
      author={Peter Teichner},
      year={1996},
      eprint={math/9612216},
      archivePrefix={arXiv},
      primaryClass={math.GT},
      url={https://arxiv.org/abs/math/9612216}, 
}

@article {neofytidis2022,
    AUTHOR = {Neofytidis, Christoforos},
     TITLE = {Topological complexity, asphericity and connected sums},
   JOURNAL = {J. Appl. Comput. Topol.},
  FJOURNAL = {Journal of Applied and Computational Topology},
    VOLUME = {9},
      YEAR = {2025},
    NUMBER = {1},
     PAGES = {Paper No. 10},
      ISSN = {2367-1726},
   MRCLASS = {55M30},
  MRNUMBER = {4879048},
       DOI = {10.1007/s41468-025-00205-z},
       URL = {https://doi.org/10.1007/s41468-025-00205-z},
}

@article {EilenbergGanea,
    AUTHOR = {Eilenberg, Samuel and Ganea, Tudor},
     TITLE = {On the {L}usternik-{S}chnirelmann category of abstract groups},
   JOURNAL = {Ann. of Math. (2)},
  FJOURNAL = {Annals of Mathematics. Second Series},
    VOLUME = {65},
      YEAR = {1957},
     PAGES = {517--518},
      ISSN = {0003-486X},
   MRCLASS = {55.0X},
  MRNUMBER = {85510},
MRREVIEWER = {J.\ C.\ Moore},
       DOI = {10.2307/1970062},
       URL = {https://doi.org/10.2307/1970062},
}

@book {Bieri,
    AUTHOR = {Bieri, Robert},
     TITLE = {Homological dimension of discrete groups},
    SERIES = {Queen Mary College Mathematics Notes},
   EDITION = {Second},
 PUBLISHER = {Queen Mary College, Department of Pure Mathematics, London},
      YEAR = {1981},
     PAGES = {iv+198},
   MRCLASS = {20J05 (18G20 57P10)},
  MRNUMBER = {715779},
}

@article {Kotschick-Neofytidis,
    AUTHOR = {Kotschick, D. and Neofytidis, C.},
     TITLE = {On three-manifolds dominated by circle bundles},
   JOURNAL = {Math. Z.},
  FJOURNAL = {Mathematische Zeitschrift},
    VOLUME = {274},
      YEAR = {2013},
    NUMBER = {1-2},
     PAGES = {21--32},
      ISSN = {0025-5874,1432-1823},
   MRCLASS = {57M05 (57M12 57M50)},
  MRNUMBER = {3054316},
MRREVIEWER = {Baris\ Coskunuzer},
       DOI = {10.1007/s00209-012-1055-3},
       URL = {https://doi.org/10.1007/s00209-012-1055-3},
}

@article {MescherSC,
    AUTHOR = {Mescher, Stephan},
     TITLE = {Spherical complexities with applications to closed geodesics},
   JOURNAL = {Algebr. Geom. Topol.},
  FJOURNAL = {Algebraic \& Geometric Topology},
    VOLUME = {21},
      YEAR = {2021},
    NUMBER = {2},
     PAGES = {1021--1074},
      ISSN = {1472-2747,1472-2739},
   MRCLASS = {55S40 (58E05 58E10)},
  MRNUMBER = {4250520},
MRREVIEWER = {Samuel\ B.\ Smith},
       DOI = {10.2140/agt.2021.21.1021},
       URL = {https://doi.org/10.2140/agt.2021.21.1021},
}

@article {COEHN2019,
    AUTHOR = {Cohen, Daniel C. and Vandembroucq, Lucile},
     TITLE = {Motion planning in connected sums of real projective spaces},
   JOURNAL = {Topology Proc.},
  FJOURNAL = {Topology Proceedings},
    VOLUME = {54},
      YEAR = {2019},
     PAGES = {323--334},
      ISSN = {0146-4124},
   MRCLASS = {55M30 (70E60)},
  MRNUMBER = {3954806},
MRREVIEWER = {Jose M. Garc\'{\i}a Calcines},
}

@article {Scott83,
    AUTHOR = {Scott, Peter},
     TITLE = {The geometries of {$3$}-manifolds},
   JOURNAL = {Bull. London Math. Soc.},
  FJOURNAL = {The Bulletin of the London Mathematical Society},
    VOLUME = {15},
      YEAR = {1983},
    NUMBER = {5},
     PAGES = {401--487},
      ISSN = {0024-6093},
   MRCLASS = {57N10 (22E10 53C20)},
  MRNUMBER = {705527},
MRREVIEWER = {John Hempel},
       DOI = {10.1112/blms/15.5.401},
       URL = {https://doi.org/10.1112/blms/15.5.401},
}

@inproceedings {CoringSiefert,
    AUTHOR = {Bryden, J. and Zvengrowski, P.},
     TITLE = {The cohomology ring of the orientable {S}eifert manifolds.
              {II}},
 BOOKTITLE = {Proceedings of the {P}acific {I}nstitute for the
              {M}athematical {S}ciences {W}orkshop ``{I}nvariants of
              {T}hree-{M}anifolds'' ({C}algary, {AB}, 1999)},
   JOURNAL = {Topology Appl.},
  FJOURNAL = {Topology and its Applications},
    VOLUME = {127},
      YEAR = {2003},
    NUMBER = {1-2},
     PAGES = {213--257},
      ISSN = {0166-8641},
   MRCLASS = {57M27 (20J05 57N10 57N65)},
  MRNUMBER = {1953328},
MRREVIEWER = {John G. Ratcliffe},}

@article {GrantCweights,
    AUTHOR = {Farber, Michael and Grant, Mark},
     TITLE = {Robot motion planning, weights of cohomology classes, and
              cohomology operations},
   JOURNAL = {Proc. Amer. Math. Soc.},
  FJOURNAL = {Proceedings of the American Mathematical Society},
    VOLUME = {136},
      YEAR = {2008},
    NUMBER = {9},
     PAGES = {3339--3349},
      ISSN = {0002-9939},
   MRCLASS = {55S99 (55P99 68T40)},
  MRNUMBER = {2407101},
MRREVIEWER = {Yuli B. Rudyak},}

@article {RUD2010,
    AUTHOR = {Rudyak, Yuli B.},
     TITLE = {On higher analogs of topological complexity},
   JOURNAL = {Topology Appl.},
  FJOURNAL = {Topology and its Applications},
    VOLUME = {157},
      YEAR = {2010},
    NUMBER = {5},
     PAGES = {916--920},
      ISSN = {0166-8641},
   MRCLASS = {55M30 (57N65 68T40)},
  MRNUMBER = {2593704},
MRREVIEWER = {Darryl McCullough},
       DOI = {10.1016/j.topol.2009.12.007},
       URL = {https://doi.org/10.1016/j.topol.2009.12.007},
}

@incollection {GrantSymm,
    AUTHOR = {Farber, Michael and Grant, Mark},
     TITLE = {Symmetric motion planning},
 BOOKTITLE = {Topology and robotics},
    SERIES = {Contemp. Math.},
    VOLUME = {438},
     PAGES = {85--104},
 PUBLISHER = {Amer. Math. Soc., Providence, RI},
      YEAR = {2007},
   MRCLASS = {55S40 (55N91 68T40 68U05 70E60)},
  MRNUMBER = {2359031},}

@article {FarberTC,
    AUTHOR = {Farber, Michael},
     TITLE = {Topological complexity of motion planning},
   JOURNAL = {Discrete Comput. Geom.},
  FJOURNAL = {Discrete \& Computational Geometry. An International Journal
              of Mathematics and Computer Science},
    VOLUME = {29},
      YEAR = {2003},
    NUMBER = {2},
     PAGES = {211--221},
      ISSN = {0179-5376},
   MRCLASS = {68T40 (55M30 68U05)},
  MRNUMBER = {1957228},
      }

@article {grantfibsymm,
    AUTHOR = {Grant, Mark},
     TITLE = {Topological complexity, fibrations and symmetry},
   JOURNAL = {Topology Appl.},
  FJOURNAL = {Topology and its Applications},
    VOLUME = {159},
      YEAR = {2012},
    NUMBER = {1},
     PAGES = {88--97},
      ISSN = {0166-8641},
   MRCLASS = {55M30 (57S15)},
  MRNUMBER = {2852952},
MRREVIEWER = {Jose M. Garc\'{\i}a-Calcines},}

@article {Svarc61,
    AUTHOR = {\v{S}varc, A. S.},
     TITLE = {The genus of a fibered space},
   JOURNAL = {Trudy Moskov. Mat. Ob\v{s}\v{c}.},
  FJOURNAL = {Trudy Moskovskogo Matemati\v{c}eskogo Ob\v{s}\v{c}estva},
    VOLUME = {10},
      YEAR = {1961},
     PAGES = {217--272},
      ISSN = {0134-8663},
   MRCLASS = {55.50},
  MRNUMBER = {0154284},
MRREVIEWER = {D. W. Kahn},
}

@book {Seifert1933,
    AUTHOR = {Seifert, Herbert and Threlfall, William},
     TITLE = {Seifert and {T}hrelfall: a textbook of topology},
    SERIES = {Pure and Applied Mathematics},
    VOLUME = {89},
 PUBLISHER = {Academic Press, Inc.,
              New York-London},
      YEAR = {1980},
     PAGES = {xvi+437},
      ISBN = {0-12-634850-2},
   MRCLASS = {55-02 (57N12)},
  MRNUMBER = {575168},
MRREVIEWER = {Ross Geoghegan},
}

@article {RUDIY2014,
    AUTHOR = {Basabe, Ibai and Gonz\'{a}lez, Jes\'{u}s and Rudyak, Yuli B. and
              Tamaki, Dai},
     TITLE = {Higher topological complexity and its symmetrization},
   JOURNAL = {Algebr. Geom. Topol.},
  FJOURNAL = {Algebraic \& Geometric Topology},
    VOLUME = {14},
      YEAR = {2014},
    NUMBER = {4},
     PAGES = {2103--2124},
      ISSN = {1472-2747},
   MRCLASS = {55M30 (55R80)},
  MRNUMBER = {3331610},
MRREVIEWER = {Dirk Sch\"{u}tz},
       DOI = {10.2140/agt.2014.14.2103},
       URL = {https://doi.org/10.2140/agt.2014.14.2103},
}

@article {Drani,
    AUTHOR = {Dranishnikov, Alexander and Sadykov, Rustam},
     TITLE = {The topological complexity of the free product},
   JOURNAL = {Math. Z.},
  FJOURNAL = {Mathematische Zeitschrift},
    VOLUME = {293},
      YEAR = {2019},
    NUMBER = {1-2},
     PAGES = {407--416},
      ISSN = {0025-5874},
   MRCLASS = {55M30 (20J99)},
  MRNUMBER = {4002282},
MRREVIEWER = {Jose M. Garc\'{\i}a-Calcines}}

@article{lee2001seifertmanifolds,
  title={Seifert manifolds},
  author={Lee, Kyung Bai and Raymond, Frank},
  journal={arXiv preprint math/0108188},
  year={2001}
}

@article {catconnsum,
    AUTHOR = {Dranishnikov, Alexander and Sadykov, Rustam},
     TITLE = {On the {LS}-category and topological complexity of a connected
              sum},
   JOURNAL = {Proc. Amer. Math. Soc.},
  FJOURNAL = {Proceedings of the American Mathematical Society},
    VOLUME = {147},
      YEAR = {2019},
    NUMBER = {5},
     PAGES = {2235--2244},
      ISSN = {0002-9939},
   MRCLASS = {55M30 (55R05 57N65)},
  MRNUMBER = {3937697},
MRREVIEWER = {Jose M. Garc\'{\i}a-Calcines},
       DOI = {10.1090/proc/14288},
       URL = {https://doi.org/10.1090/proc/14288},
}

@book {Lennox,
    AUTHOR = {Lennox, John C. and Robinson, Derek J. S.},
     TITLE = {The theory of infinite soluble groups},
    SERIES = {Oxford Mathematical Monographs},
 PUBLISHER = {The Clarendon Press, Oxford University Press, Oxford},
      YEAR = {2004},
     PAGES = {xvi+342},
      ISBN = {0-19-850728-3},
   MRCLASS = {20F16 (20-02 20E15)},
  MRNUMBER = {2093872},
MRREVIEWER = {Patrizia Longobardi},
       DOI = {10.1093/acprof:oso/9780198507284.001.0001},
       URL = {https://doi.org/10.1093/acprof:oso/9780198507284.001.0001},
}

@incollection {cwtFadellHusseini,
    AUTHOR = {Fadell, Edward and Husseini, Sufian},
     TITLE = {Category weight and {S}teenrod operations},
      NOTE = {Papers in honor of Jos\'{e} Adem (Spanish)},
   JOURNAL = {Bol. Soc. Mat. Mexicana (2)},
  FJOURNAL = {Bolet\'{\i}n de la Sociedad Matem\'{a}tica Mexicana. Segunda Serie},
    VOLUME = {37},
      YEAR = {1992},
    NUMBER = {1-2},
     PAGES = {151--161},
   MRCLASS = {55M30 (55S05)},
  MRNUMBER = {1317569},
MRREVIEWER = {Edgar H. Brown, Jr.},}

@Article{cat3mfds,
 Author = {G{\'o}mez-Larra{\~n}aga, J. C. and Gonz{\'a}lez-Acu{\~n}a, F.},
 Title = {Lusternik-{Schnirelmann} category of 3-manifolds},
 FJournal = {Topology},
 Journal = {Topology},
 ISSN = {0040-9383},
 Volume = {31},
 Number = {4},
 Pages = {791--800},
 Year = {1992},
 Language = {English},
 DOI = {10.1016/0040-9383(92)90009-7},
 zbMATH = {167784},
 Zbl = {0778.55002}
}

@article {Surveysf,
    AUTHOR = {Pr\'{e}aux, Jean-Philippe},
     TITLE = {A survey on {S}eifert fiber space theorem},
   JOURNAL = {ISRN Geom.},
  FJOURNAL = {ISRN Geometry},
      YEAR = {2014},
     PAGES = {Art. ID 694106, 9},
      ISSN = {2090-6307,2090-6315},
   MRCLASS = {57-02},
  MRNUMBER = {3178943},
       DOI = {10.1155/2014/694106},
       URL = {https://doi.org/10.1155/2014/694106},
}

@article {htcsurfaces,
    AUTHOR = {Gonz\'{a}lez, Jes\'{u}s and Guti\'{e}rrez, B\'{a}rbara and
              Guti\'{e}rrez, Darwin and Lara, Adriana},
     TITLE = {Motion planning in real flag manifolds},
   JOURNAL = {Homology Homotopy Appl.},
  FJOURNAL = {Homology, Homotopy and Applications},
    VOLUME = {18},
      YEAR = {2016},
    NUMBER = {2},
     PAGES = {359--275},
      ISSN = {1532-0073,1532-0081},
   MRCLASS = {55M30 (57T15)},
  MRNUMBER = {3576004},
MRREVIEWER = {Yuli\ B.\ Rudyak},
       DOI = {10.4310/HHA.2016.v18.n2.a20},
       URL = {https://doi.org/10.4310/HHA.2016.v18.n2.a20},
}

@article {daundkar2023group,
    AUTHOR = {Daundkar, Navnath},
     TITLE = {Group actions and higher topological complexity of lens
              spaces},
   JOURNAL = {J. Appl. Comput. Topol.},
  FJOURNAL = {Journal of Applied and Computational Topology},
    VOLUME = {8},
      YEAR = {2024},
    NUMBER = {7},
     PAGES = {2051--2067},
      ISSN = {2367-1726,2367-1734},
   MRCLASS = {55M30 (57N65 57S15)},
  MRNUMBER = {4817708},
MRREVIEWER = {Nicholas\ Wawrykow},
       DOI = {10.1007/s41468-024-00171-y},
       URL = {https://doi.org/10.1007/s41468-024-00171-y},
}

\end{document}